\documentclass[12pt]{amsart}
\usepackage{amscd}
\usepackage{amsmath}
\usepackage{amssymb}
\usepackage{amsthm}
\newtheorem{thm}{Theorem}[section]
\newtheorem{lemma}[thm]{Lemma}
\newtheorem{prop}[thm]{Proposition}
\newtheorem{cor}[thm]{Corollary}

\newtheorem*{fact}{Fact}
\theoremstyle{definition}

\newtheorem{exam}[thm]{Example}

\newtheorem*{acknowledgement}{Ackowledgement}
\theoremstyle{remark}
\newtheorem{remark}[thm]{Remark}
\numberwithin{equation}{section}

\DeclareMathOperator{\pd}{pd}

\DeclareMathOperator{\cohodim}{cd}

\DeclareMathOperator{\height}{height}

\DeclareMathOperator{\ara}{ara}

\DeclareMathOperator{\depth}{depth}

\newcommand{\frm}{{\mathfrak m}}

\begin{document}
\title[Schmitt--Vogel type lemma for reductions]
{Schmitt--Vogel type lemma for reductions}
\author[K. Kimura]{Kyouko Kimura}
\address[Kyouko Kimura]{Department of Pure and Applied Mathematics, 
         Graduate School of Information Science and Technology, 
         Osaka University, Toyonaka, Osaka 560--0043, Japan}
\email{kimura@math.sci.osaka-u.ac.jp}
\author[N. Terai]{Naoki Terai}
\address[Naoki Terai]{Department of Mathematics, 
         Faculty of Culture and Education, 
         Saga University, Saga 840--8502, Japan}
\email{terai@cc.saga-u.ac.jp}
\author[K. Yoshida]{Ken-ichi Yoshida}
\address[Ken-ichi Yoshida]{Graduate School of Mathematics, Nagoya University, 
         Nagoya 464--8602, Japan}
\email{yoshida@math.nagoya-u.ac.jp}
\date{\today}
\begin{abstract}
The lemma given by Schmitt and Vogel is an important tool 
in the study of arithmetical rank of squarefree monomial ideals. 
In this paper, we give a Schmitt--Vogel type lemma for reductions
as an analogous result.  
\end{abstract}
\maketitle
\setcounter{section}{-1}
\section{Introduction}
Throughout this paper, let $R$ be a commutative Noetherian ring 
with non-zero identity. 
Let $I$ be an ideal of $R$. 
Then the \textit{arithmetical rank} of $I$ is defined by 
\begin{displaymath}
  \ara I := \min \{ r \; : \; \text{there exist some $a_1, \ldots, a_r \in R$ 
    such that $\sqrt{(a_1, \ldots, a_r)} = \sqrt{I}$} \}. 
\end{displaymath}
If $\sqrt{(a_1, \ldots, a_r)} =\sqrt{I}$ holds, then we say that 
$a_1, \ldots, a_r$ generate $I$ \textit{up to radical}. 
\par 
Assume that $R$ is a polynomial ring over a field $K$ and 
$I$ is generated by squarefree monomials. 
Then we have the following inequalities: 
\begin{displaymath}
  \height I \le \pd_R R/I = \cohodim (I) 
  \le \ara I \le \mu (I), 
\end{displaymath}
where $\height I$ (resp.\  $\pd_R R/I$, $\cohodim (I)$, $\mu(I)$) denotes 
the height of $I$ (resp.\  the projective dimension of $R/I$ over $R$, 
the cohomological dimension of $I$, the minimal number of generators of $I$); 
see e.g. \cite{KTY}.  
Many researchers, e.g.  Barile \cite{Bar96, Bar05, Bar07, Bar08-3, Bar0611}, 
Schmitt and Vogel \cite{SV} 
and the authors  \cite{KTY, KTY2}
have proved $\ara I = \pd_R R/I$  
using the following lemma given by Schmitt and Vogel \cite{SV} or 
its generalizations. 

\begin{fact}[\textbf{Schmitt and Vogel} \textrm{\cite[Lemma, p.\  249]{SV}}]
Let $\mathcal{P}$ be a finite subset of $R$, 
and let $I$ be the ideal generated by $\mathcal{P}$.  
Let $r \ge 0$ be an integer. 
Assume that there exist subsets 
${\mathcal{P}}_0, {\mathcal{P}}_1, \ldots, {\mathcal{P}}_r$ of $\mathcal{P}$ 
such that the following conditions are satisfied$:$ 
\begin{enumerate}
 \item[(i)] $\mathcal{P}={\mathcal{P}}_0 \cup {\mathcal{P}}_1 \cup \cdots \cup {\mathcal{P}}_r$. 
 \item[(ii)] $\sharp {\mathcal{P}}_0 = 1$. 
 \item[(iii)] For each $\ell$ $(0 < \ell \leq r)$ 
    and for every $a, a'' \in {\mathcal{P}}_{\ell}$ with $a \neq a''$, 
    there exist an integer ${\ell}'$ 
    $(0 \leq {\ell}' < \ell)$, and elements $a' \in {\mathcal{P}}_{{\ell}'}$,  
    such that $a a'' \in (a')$. 
\end{enumerate}
If we set
  \begin{displaymath}
    g_{\ell} = \sum_{a \in {\mathcal{P}}_{\ell}} a, 
    \quad \ell = 0, 1, \ldots, r,  
  \end{displaymath}
  then $\sqrt{I} = \sqrt{(g_0,g_1,\ldots,g_r)}$. 
\end{fact}

\par \vspace{2mm}
An ideal $J \subset I$ is said to be a \textit{reduction} if there exists 
some integer $s \geq 1$ such that $I^{s+1} = JI^{s}$ holds. 
When this is the case, $\sqrt{J} = \sqrt{I}$ holds. 
If $J$ is minimal among reductions of $I$ with respect to inclusion, then it is said to be a 
\textit{minimal reduction} of $I$. 
Let $R$ be a polynomial ring over a field $K$ and $I$ a homogeneous ideal of $R$, 
or let $R$ be a local ring with unique maximal ideal $\frm$ and $K=R/\frm$ and 
$I$ an ideal of $R$.  
If $K$ is infinite, then for any (homogeneous) ideal $I$, 
we can take a minimal reduction $J$ of $I$  
and the minimal number of generators of $J$ 
is independent of the choice of $J$; see  \cite{NorthRees}. 
The number of generators of $J$ is called the \textit{analytic spread} of $I$ 
(denoted by $\ell(I)$) and it gives   
an upper bound for $\ara I$. 
In the commutative ring theory, the minimal reduction 
plays an important role because it admits the same integral closure 
as the original ideal.  
Moreover, the analytic spread is equal to the Krull dimension of 
the fiber cone $F(I)= \bigoplus_{n \ge 0} I^n/\frm I^n$ of $I$ 
in a local ring $(R,\frm)$, and hence it is an important invariant. 

\par \vspace{2mm}     
The main purpose of this note is to give an analogous result 
of the lemma due to Schmitt and Vogel \cite[Lemma, p.\ 249]{SV} 
for reductions; see Theorem \ref{claim:SV_red}.   
For instance, let us consider the following monomial ideal 
in a suitable polynomial ring $R$: 
\begin{equation}
  \label{eq:ci_dual_intro}
  I = (x_{11}, \ldots, x_{1 h_1}) 
      \cap \cdots \cap (x_{q1}, \ldots, x_{q h_q}).  
\end{equation}
In order to give an upper bound for $\cohodim (I)$, 
Schenzel and Vogel \cite{Sche-Vo} computed $\depth R/I^{\ell}$ for all $\ell \ge 1$, 
and proved 
\[
 \cohodim (I) \le \ell(I) \le \depth R - \inf_{\ell} \depth R/I^{\ell} 
= \sum_{i=1}^q h_i -q +1 \ \big(= \pd_R R/I\big), 
\]
where the second inequality is known as Burch's inequality. 
On the other hand,  Schmitt and Vogel \cite{SV} constructed $\pd_R R/I$ generators 
up to radical using their lemma. 
By using Theorem  \ref{claim:SV_red}  instead of their lemma, we can provide a minimal reduction 
with $\pd_R R/I$ generators; see Example \ref{claim:ci_dual}.

\par 
In Section 2, we prove the main theorem. 
In order to do that, we give analogous results 
(see Theorems \ref{claim:Bar07_red}, \ref{claim:Bar96_red})
of generalizations of the lemma due to Barile for reductions,
and prove them.

\section{Schmitt--Vogel type lemma for reductions and its application}
\label{sec:SV_red}

The following theorem is the main result in this paper, 
which gives an analogous result of \cite[Lemma, p.\  249]{SV} proved by Schmitt and Vogel. 
Note that the theorem immediately follows from 
Theorem \ref{claim:Bar07_red}, which will be proved in Section \ref{sec:Bar_red}. 

\begin{thm}[\textbf{Schmitt--Vogel type lemma for reductions}]
  \label{claim:SV_red}
Let $\mathcal{P}$ be a finite subset of $R$, and let $I$ be the ideal generated by $\mathcal{P}$.  
Let $r \ge 0$ be an integer. 
Assume that there exist subsets 
${\mathcal{P}}_0, {\mathcal{P}}_1, \ldots, {\mathcal{P}}_r$ of $\mathcal{P}$ 
such that the following conditions are satisfied$:$ 
\begin{enumerate}
 \item[(SV1)] $\mathcal{P}={\mathcal{P}}_0 \cup {\mathcal{P}}_1 \cup \cdots \cup {\mathcal{P}}_r$. 
 \item[(SV2)] $\sharp {\mathcal{P}}_0 = 1$. 
 \item[(SV3)] For each $\ell$ $(0 < \ell \leq r)$ 
    and for every $a, a'' \in {\mathcal{P}}_{\ell}$ with $a \neq a''$, 
    there exist an integer ${\ell}'$ 
    $(0 \leq {\ell}' < \ell)$, and elements $a' \in {\mathcal{P}}_{{\ell}'}$,  
    $b \in I$ such that $a a'' = a' b$. 
\end{enumerate}
If we set
  \begin{displaymath}
    g_{\ell} = \sum_{a \in {\mathcal{P}}_{\ell}} a, 
    \quad \ell = 0, 1, \ldots, r,  
  \end{displaymath}
  then $J = (g_0, g_1, \ldots, g_r)$ is a reduction of $I$. 
\end{thm}


\par 
We now restrict our attention to the following case: 
$R$ is a polynomial ring over a field $K$ 
and $I$ is a squarefree monomial ideal of $R$. 
In this case, 
as an application of the above theorem, we have the following result. 

\begin{cor}
\label{claim:squareSV_red}
Let $R$ be a polynomial ring and $I$ a squarefree monomial ideal of $R$. 
Assume that there exist finite subsets $\mathcal{P}_0,\ldots,\mathcal{P}_r$ of $R$ 
satisfying the assumptions in Theorem $\ref{claim:SV_red}$ for $r = \pd_R R/I-1$.  
Then $(g_0,g_1,\ldots,g_r)$ is a minimal reduction of $I$, and 
$\ell(I) = \ara I = \pd_R R/I =r+1$. 
\end{cor}

\begin{proof}
Since $I$ is a squarefree monomial ideal, 
we have 
\[
 r+1=\pd_R R/I = \cohodim(I) \le \ara I \le \ell(I). 
\]
On the other hand, Theorem \ref{claim:SV_red} implies  
$\ell(I) \le r+1$. 
Hence we get the desired assertion. 
\end{proof}

\par \vspace{2mm}
We can apply our results to Alexander dual of complete intersection monomial ideals; see below.

\begin{exam}[\textbf{Alexander dual of complete intersection monomial ideals}]
  \label{claim:ci_dual}
Let $I \subseteq R$ be a squarefree monomial ideal of the following shape:  
\begin{equation}
  \label{eq:ci_dual}
  (x_{11}, \ldots, x_{1 h_1}) 
    \cap \cdots \cap (x_{q1}, \ldots, x_{q h_q}),    
\end{equation}
where $R=K[x_{11}, \ldots, x_{1 h_1},\ldots,x_{q1}, \ldots, x_{q h_q}]$ is a polynomial ring 
over a field $K$.  
Note that $I$ can be regarded as the Alexander dual of complete intersection 
monomial ideal $(x_{11}\cdots x_{1h_1},\ldots,x_{q1}\cdots x_{qh_q})$ if $h_1,\ldots,h_q \ge 2$. 
\par 
Set $r = h_1+\cdots + h_q - q$ and 
  \begin{displaymath}
    g_{\ell} = \sum_{{\ell}_1 + \cdots + {\ell}_q = \ell} 
      x_{1 {\ell}_1} x_{2 \ell_2} \cdots x_{q {\ell}_q}, \quad \ell = 0, 1, \ldots, r. 
  \end{displaymath}
  Then $(g_0, g_1, \ldots, g_r)$ is a minimal reduction of $I$. 
  In particular, 
  \begin{displaymath}
    \ell (I) = \ara I = \pd_R R/I = \sum_{i=1}^q h_i - q +1. 
  \end{displaymath}
\end{exam}

\begin{proof}
It is known that 
  \begin{displaymath}
     r+1 = \pd_R R/I = \ara I \le \ell(I);
  \end{displaymath}
see e.g. \cite[Theorem]{SV} or \cite[Section 5]{KTY}. 
\par 
For each $\ell =0,1,\ldots, r$, we set 
  \begin{displaymath}
    {\mathcal{P}}_{\ell} 
    = \left\{ x_{1 {\ell}_1} \cdots x_{q {\ell}_q} \; : \; 
      1 \leq {\ell}_j \le h_j, \; {\ell}_1+\cdots + {\ell}_q 
      = \ell +q \right\}.
  \end{displaymath}
Then $I$ is generated by all monomials in $\mathcal{P}_0 \cup \cdots \cup \mathcal{P}_r$, 
and $\mathcal{P}_0$ consists of only one element $x_{11}\cdots x_{q1}$. 
Thus it suffices to show that if $a$, $a'' \in {\mathcal{P}}_{\ell}$ 
  with $a \neq a''$ then there exist $a' \in {\mathcal{P}}_{\ell'}$ 
  for some $\ell'< \ell$  and $b \in I$ such that $aa'' = a'b$. 
  Write 
  \begin{displaymath}
  a = x_{1 i_1} x_{2 i_2} \cdots x_{q i_q}, \qquad 
  a'' = x_{1 j_1} x_{2 j_2}\cdots x_{q j_q}, 
  \end{displaymath}
  where $i_1+ \cdots + i_q = j_1 + \cdots + j_q = \ell +q$. 
  As $a \neq a''$, there exists an integer $k$ $(1 \le k \le q)$ 
  such that $i_k > j_k$. 
  We may assume that $k=1$ without loss of generality. 
  Then if we set 
  \begin{displaymath}
    a'  = a \cdot \frac{x_{1j_1}}{x_{1i_1}}
       = x_{1j_1}x_{2i_2}\cdots x_{qi_q}, 
    \qquad 
    b = a'' \cdot \frac{x_{1i_1}}{x_{1j_1}}
      = x_{1i_1} x_{2j_2} \cdots x_{qj_q} 
    \in I,
  \end{displaymath}
  then $aa'' = a'b$ and $a' \in \mathcal{P}_{\ell'}$, where 
  \begin{displaymath}
    \ell' = j_1 + i_2+\cdots + i_q -q < i_1+ i_2 + \cdots + i_q -q = \ell.
  \end{displaymath}
  Hence we can apply Corollary \ref{claim:squareSV_red}.
\end{proof}

\begin{remark}
  \label{rmk:SV_ara}
We use the same notation as in Example \ref{claim:ci_dual}. 
Schmitt and Vogel \cite{SV} proved $\ara I = \pd_R R/I$ 
by showing $\sqrt{(g_0, g_1, \ldots, g_r)} = \sqrt{I}$. 
Thus the above example gives an improvement of their result. 
\end{remark}

\par
We can generalize Example \ref{claim:ci_dual} as follows.  

\begin{prop}
  \label{claim:ci_dual+}
For each $i=1,2,\ldots,s$, 
let $I_i$ be a squarefree monomial ideal of the shape 
 $(\ref{eq:ci_dual})$:
\begin{displaymath}
    I_i = (x_{11}^{(i)}, \ldots, x_{1 h_1^{(i)}}^{(i)}) \cap \cdots \cap 
          (x_{q^{(i)} 1}^{(i)}, \ldots, x_{q^{(i)} h_{q^{(i)}}^{(i)}}^{(i)}). 
\end{displaymath}
Let $G(I_i)$ be the minimal set 
of monomial generators of $I_i$. 
Suppose that there are no variables which appear in both 
$G (I_i)$ and $G (I_j)$ 
for each $i, j$ with $i \neq j$. 
For $I_i$, set $g_{\ell}^{(i)}$ as in Example $\ref{claim:ci_dual}$. 
 Then 
\[
 (g_{\ell}^{(i)} \,:\, i=1,\ldots,s, \ell = 0,1,\ldots, h_1^{(i)}+\cdots+ h_{q^{(i)}}^{(i)}-q^{(i)})
\] 
is a minimal reduction of $I_1 + \cdots + I_s$. 
In particular, $\ell(I_1+\cdots +I_s) = \ell(I_1)+\cdots + \ell(I_s)$. 
\end{prop}

\par
In order to prove Proposition \ref{claim:ci_dual+}, it is enough to show the following lemma. 

\begin{lemma}
\label{claim:sum_ideal}
Let $R$, $S$ be polynomial rings over a field $K$ with no common variables, 
and put $T=R \otimes_K S$. 
Let $I \subseteq R$ $($resp.\  $J \subseteq S$ $)$ be a squarefree monomial ideal.    
Then:
\begin{enumerate}
 \item $\pd_T T/(IT+JT) = \pd_{R} R/I + \pd_{S} S/J$. 
 \item Assume that $\mathcal{P}_0,\mathcal{P}_1, \ldots,\mathcal{P}_{r} \subseteq R$  
$($resp.\  $\mathcal{Q}_0,\mathcal{Q}_1,\ldots,\mathcal{Q}_{s} \subseteq S$ $)$ 
satisfies $(SV1)$, $(SV2)$ and $(SV3)$ in Theorem $\ref{claim:SV_red}$.   
Then 
$\mathcal{P}_0,\mathcal{P}_1,\ldots,\mathcal{P}_{r}, \mathcal{Q}_0,\mathcal{Q}_1,\ldots,\mathcal{Q}_{s}$  
also satisfies the same conditions as finite subsets of $T$. 
\end{enumerate} 
\end{lemma}

\begin{proof}
(1) Let $F_{\bullet}$ (resp.\  $G_{\bullet}$) be a minimal free resolution 
of $R/I$ over $R$ (resp.\  $S/J$ over $S$). 
Then $F_{\bullet} \otimes_K G_{\bullet}$ is a minimal free resolution of $T/(IT+JT)$.  
Thus we have $\pd_T T/(IT+JT) = \pd_{R} R/I + \pd_{S} S/J$. 
\par \vspace{2mm}
(2) It is clear by definition. 
\end{proof}

\begin{remark}
We use the same notation as in Lemma \ref{claim:sum_ideal}.  
Then it is easy to see that $\ara(IT+JT) \le \ara I + \ara J$ holds. 
If both $\ara I = \pd_R R/I$ and $\ara J = \pd_S S/J$ hold, then 
the equality holds. 
But we do \textit{not} know whether it is always true. 
Moreover, it seems that a similar result holds for analytic spreads, 
but we do \textit{not} have any proof in general. 
\end{remark}

\section{Proof of the theorem}
\label{sec:Bar_red}

In this section, we prove Theorem \ref{claim:SV_red}, which is an analogous 
result of the lemma by Schmitt--Vogel for reductions.  
But the lemma has been generalized by Barile \cite{Bar96, Bar07}; 
see also \cite{Bar0611} for another version.  
So, in this section, we prove analogous results 
for two generalizations by Barile; see Theorems \ref{claim:Bar07_red}, 
\ref{claim:Bar96_red}.

%

\par \vspace{2mm}
The following theorem gives an analogous result for Barile \cite[Lemma 2.1]{Bar07}, 
which is a generalization of Theorem \ref{claim:SV_red}. 
 
\begin{thm}
  \label{claim:Bar07_red} 
Let $\mathcal{P} \subset R$ be a finite subset, and put $I= (\mathcal{P})$.  
Let ${\mathcal{P}}_0, {\mathcal{P}}_1, \ldots, {\mathcal{P}}_r$ be 
subsets of $\mathcal{P}$. 
Assume that the following conditions$:$ 
  \begin{enumerate}
  \item[(B1)] $\mathcal{P}= \mathcal{P}_0 \cup \mathcal{P}_1 \cup \cdots \cup  {\mathcal{P}}_r$. 
  \item[(B2)] $\sharp {\mathcal{P}}_0 = 1$. 
  \item[(B3)] For each $\ell$ $(0 < \ell \leq r)$ 
    and for every $a, a'' \in {\mathcal{P}}_{\ell}$ with $a \ne a''$, there exists an integer
    $m \ge 1$ such that 
    $(a a'')^m \in (\mathcal{P}_0 \cup \cdots \cup \mathcal{P}_{\ell - 1})I^{2m-1}$. 
  \end{enumerate}
Set 
\begin{displaymath}
    g_{\ell} = \sum_{a \in {\mathcal{P}}_{\ell}} a, 
    \quad \ell = 0, 1, \ldots, r. 
  \end{displaymath}
  Then $J = (g_0, g_1, \ldots, g_r)$ is a reduction of $I$. 
\end{thm}

\begin{remark}
  \label{rmk:Diff_Bar07}
  The difference between Theorem \ref{claim:Bar07_red} and the original result 
  of Barile \cite{Bar07} is in the condition (B3). 
  The condition of the original result corresponding to (B3) is 
  \begin{enumerate}
  \item[(B3)'] For each $\ell$ $(0 < \ell \leq r)$ 
    and for every $a, a'' \in {\mathcal{P}}_{\ell}$ with $a \ne a''$, there exists 
    an integer $m \ge 1$ such that 
    $(a a'')^{m} \in (\mathcal{P}_0 \cup \cdots \cup {\mathcal{P}}_{\ell-1})$. 
  \end{enumerate}
\end{remark}

\begin{proof}[Proof of Theorem $\ref{claim:Bar07_red}$]
Since $J \subseteq I$, it suffices to show $I^{s+1} \subset JI^s$. 
In order to do that, 
we set $\sharp {\mathcal{P}}_{\ell} = c_{\ell}$ and 
$I_{\ell} = (\mathcal{P}_0 \cup \cdots \cup \mathcal{P}_{\ell})$  
for each $\ell =0,1,\ldots,r$.
Moreover, for each $\ell$, 
we can choose an integer $m_{\ell} \ge 1$ such that 
\[
(a a'')^{m_{\ell}} \in I_{\ell - 1} I^{2m_{\ell} - 1}
\] 
for all $a, a'' \in {\mathcal{P}}_{\ell}$ with $a \neq a''$ by assumption.  
Then it is enough to prove 
  \begin{equation}
    \label{eq:Bar07_ETS}
    I_{j}^{c_1 \cdots c_j m_1 \cdots m_{j}} 
    \subset I_{j-1}^{c_1 \cdots c_{j-1} m_1 \cdots m_{j-1}} 
            I^{(c_1 \cdots c_{j - 1} m_1 \cdots m_{j-1})
                 (c_{j} m_{j} - 1)} 
            + J I^{c_1 \cdots c_{j} m_1 \cdots m_{j} - 1}
  \end{equation}
for each $j = 0, 1, \ldots, r$. 

\par
The case of $j = 0$ is clear because 
$I_0 = ({\mathcal{P}}_0) = (g_0) \subset J$. 

\par
Now suppose $j = \ell \ge 1$ and assume that (\ref{eq:Bar07_ETS}) holds 
for every $j \le \ell-1$. 
To prove (\ref{eq:Bar07_ETS}) for $j=\ell$, it is enough to show that  
for arbitrary $c_1 \cdots c_{\ell} m_1 \cdots m_{\ell}$ elements 
(to take the same elements is allowed) 
in $\mathcal{P}_0 \cup \cdots \cup \mathcal{P}_{\ell}$, 
the product of all elements is contained in the right hand side of (\ref{eq:Bar07_ETS}). 
We divide these elements into  
$c_1 \cdots c_{\ell - 1} m_1 \cdots m_{\ell - 1}$ sequences of 
$c_{\ell} m_{\ell}$ elements, and show that the product of the elements 
in each sequence is in 
$I_{\ell - 1} I^{c_{\ell} m_{\ell} - 1} + J I^{c_{\ell} m_{\ell} - 1}$. 
\par 
In what follows, we discuss about only one sequence. 
If there exists an element of 
$\mathcal{P}_0 \cup \cdots \cup \mathcal{P}_{\ell-1}$ 
in the sequence, then it is clear that the product 
is in $I_{\ell - 1} I^{c_{\ell} m_{\ell} -1}$. 
Therefore, we may assume that all elements in the sequence are in ${\mathcal{P}}_{\ell}$. 
If we can find a pair $(a, a'')$ with $a \ne a''$ which appear
at least $m_{\ell}$ times in this sequence, then the assumption (B3) 
yields that the product of all elements in the sequence is contained in 
$I_{\ell - 1} I^{c_{\ell} m_{\ell} - 1}$. 
Otherwise, we pick up an element $a_1$ the number of times (say, $d$) 
which appears in the sequence is maximal. 
Note that $d > m_{\ell}$. 
Let ${\mathcal{P}}_{\ell} = \{ a_1, a_2, \ldots, a_{c_{\ell}} \}$. 
Then the product of all elements in the sequence is 
  \begin{displaymath}
    \begin{aligned}
      a_1^d a_2^{k_2} \cdots a_{c_{\ell}}^{k_{c_{\ell}}} 
      = &a_1^{m_{\ell}} a_1^{d - m_{\ell}} 
         a_2^{k_2} \cdots a_{c_{\ell}}^{k_{c_{\ell}}} \\
      = &a_1^{m_{\ell}} 
         \bigg( g_{\ell} - \sum_{i=2}^{c_{\ell}} a_i \bigg)^{d - m_{\ell}}
         a_2^{k_2} \cdots a_{c_{\ell}}^{k_{c_{\ell}}} \\
      = &g_{\ell} \cdot (\text{the products of $c_{\ell} m_{\ell} - 1$ elements 
                         of ${\mathcal{P}}_{\ell}$}) \\
        &+ a_1^{m_{\ell}} 
           \bigg( \sum_{i=2}^{c_{\ell}} a_i \bigg)^{d - m_{\ell}} 
           a_2^{k_2} \cdots a_{c_{\ell}}^{k_{c_{\ell}}} \\
      = &g_{\ell} \cdot (\text{the products of $c_{\ell} m_{\ell} - 1$ elements 
                         of ${\mathcal{P}}_{\ell}$}) \\
        &+ \sum a_1^{m_{\ell}} 
           a_2^{k_2'} \cdots a_{c_{\ell}}^{k_{c_{\ell}}'},
    \end{aligned}
  \end{displaymath}
where $k_2 + \ldots + k_{c_{\ell}} = c_{\ell}m_{\ell} - d$ and 
$k_2' + \ldots + k'_{c_{\ell}} = (c_{\ell} -1) m_{\ell}$. 
Then there exists an integer $j$ with $2 \leq j \le c_{\ell}$ such that 
$k_j' \geq m_{\ell}$.
By a similar argument as above, the right-hand side is contained in 
$JI^{c_{\ell}m_{\ell}-1}+I_{\ell-1}I^{c_{\ell} m_{\ell}-1}$. 
Hence we have finished the proof. 
\end{proof}

\begin{proof}[Proof of Theorem $\ref{claim:SV_red}$]
Assume that $I$ satisfies (SV1),(SV2), and (SV3). 
Then it also satisfies (B1), (B2) and (B3). 
Hence the assertion immediately follows from Theorem \ref{claim:Bar07_red}. 
\end{proof}

\par \vspace{2mm}
In the proof of the following two examples, we need Theorem \ref{claim:Bar07_red} 
instead of Theorem \ref{claim:SV_red}.
 
\begin{exam}
\label{claim:hyper_exam}
Let $K$ be a field, and let $m \ge 2$ be an integer. 
Consider the hypersurface $R=K[[x,y,z]]/(x^my^m-z^{2m})$. 
For an ideal $I=(x,y,z)R$, we put 
\[
\mathcal{P}_0=\{z\},\quad \mathcal{P}_1 = \{x,y\}. 
\]  
Then since $(xy)^m = z \cdot z^{2m-1} \in (\mathcal{P}_0) I^{2m-1}$, 
we can conclude that $x+y,z$ is a (minimal) reduction by virtue of  
Theorem \ref{claim:Bar07_red}. 
But we cannot apply Theorem \ref{claim:SV_red} to this case 
because $xy \notin (z)$. 
\end{exam}

\begin{exam}
\label{claim:Binom_exam}
Let $R=K[x_1,x_2,x_3,x_4,x_5,x_6]$ be a polynomial ring over a field $K$. 
For an ideal 
\[
I=(x_1x_2+x_1x_3,\, x_1x_4,\, x_1x_5, \, x_1x_6,x_2x_5,\, x_2x_6,\, x_3x_4,\, x_3x_6,\, x_4x_5,\, x_4x_6,\,x_5x_6), 
\]
we put 
\[
\begin{array}{rclcrcl}
\mathcal{P}_0 &= & \{x_1x_6\}, & \qquad &
\mathcal{P}_1 &=& \{x_1x_5,\ x_2x_6\}, \\
\mathcal{P}_2 &=& \{x_1x_4,\ x_3x_6\}, & \qquad & 
\mathcal{P}_3 &=& \{x_2x_5,\ x_4x_6\}, \\
\mathcal{P}_4 &=& \{x_3x_4,\ x_5x_6\}, & \qquad & 
\mathcal{P}_5 &=& \{x_1x_2+x_1x_3,\ x_4x_5\}. 
\end{array}
\]
Then we can conclude that 
\[
 J = (x_1x_6,\, x_1x_5+x_2x_6, \, x_1x_4+x_3x_6,\, x_2x_5+x_4x_6,\, x_3x_4+x_5x_6,\, x_1x_2+x_1x_3+x_4x_5)
\]
is a (minimal) reduction of $I$ by Theorem \ref{claim:Bar07_red}. 
But we cannot apply Theorem \ref{claim:SV_red} because the product of 
$(x_1x_2+x_1x_3) \in \mathcal{P}_5$ and $x_4x_5 \in \mathcal{P}_5$ 
is not contained in the ideal $(a')$ for any element 
$a' \in \mathcal{P}_0 \cup \cdots \cup \mathcal{P}_4$.  
\end{exam}

\bigskip

\par
Next, we refine the result by Barile \cite[Proposition 1.1]{Bar96}
\begin{thm}
  \label{claim:Bar96_red}
Assume that $R$ is a local ring. 
Let $\mathcal{P} \subset R$ be a finite subset, and  
let ${\mathcal{P}}_0, {\mathcal{P}}_1, \ldots, {\mathcal{P}}_r$ be 
subsets of $\mathcal{P}$. 
We set $\sharp {\mathcal{P}}_{\ell} = c_{\ell}$ for all $\ell$ and $I=({\mathcal{P}})$. 
Assume that the following conditions are satisfied$:$ 
  \begin{enumerate}
  \item[(Ba1)] $\mathcal{P} = \mathcal{P}_0 \cup \mathcal{P}_1 \cup \cdots \cup {\mathcal{P}}_{r}$. 
  \item[(Ba2)] $\sharp {\mathcal{P}}_0 = 1$. 
  \item[(Ba3)] For each $\ell$ $(0 < \ell \leq r)$ with $c_{\ell} \geq 2$, 
    there exists an integer $n_{\ell}$ with $2 \leq n_{\ell} \leq c_{\ell}$ 
    such that 
    for arbitrary $n_{\ell}$ distinct elements 
    $p_1, p_2, \ldots, p_{n_{\ell}} \in {\mathcal{P}}_{\ell}$, there exist 
    an integer ${\ell}'$ with $0 \leq {\ell}' < \ell$, elements $p' \in {\mathcal{P}}_{{\ell}'}$ 
    and $b \in I^{n_{\ell} - 1}$ 
    such that $p_1 p_2 \cdots p_{n_{\ell}} = p' b$. 
  \end{enumerate}
  For $0 \leq \ell \leq r$ with $c_{\ell} = 1$, we set $n_{\ell} = 2$. 
  For each $\ell = 0, 1, \ldots, r$, let 
  $A^{({\ell})} = (a_{ij}^{({\ell})})$ be an $(n_{\ell} - 1) \times c_{\ell}$ 
  matrix with $a_{ij}^{({\ell})} \in R$. 
  Assume that all maximal minors of $A^{({\ell})}$ are unit in $R$. 
  Set 
  \begin{displaymath}
    \begin{aligned}
      {\mathcal{P}}_{\ell} 
      &= \{ p_1^{({\ell})}, p_2^{({\ell})}, 
            \ldots, p_{c_{\ell}}^{({\ell})} \}, 
         \quad 0 \leq \ell \leq r, \\
      g_i^{({\ell})} 
      &= \sum_{j=1}^{c_{\ell}} a_{ij}^{({\ell})} p_j^{({\ell})}, 
         \quad 1 \leq i \leq n_{\ell} - 1, \quad 0 \leq \ell \leq r, \\
      J &= (g_i^{({\ell})} 
         \; : \; 0 \leq \ell \leq r, \  1 \leq i \leq n_{\ell} - 1). 
    \end{aligned}
  \end{displaymath}
  Then $J$ is a reduction of $I$. 
\end{thm}

\begin{remark}
  \label{rmk:Diff_Bar96}
  The difference between Theorem \ref{claim:Bar96_red} and the original result 
  of Barile \cite{Bar96} is in the condition (Ba3). 
  The condition of the original result corresponding to (Ba3) is 
  \begin{enumerate}
  \item[(Ba3)'] For each $\ell$ $(0 < \ell \leq r)$ with $c_{\ell} \geq 2$, 
    there exists some integer $n_{\ell}$, $2 \leq n_{\ell} \leq c_{\ell}$ 
    such that 
    for arbitrary $n_{\ell}$ distinct elements 
    $p_1, p_2, \ldots, p_{n_{\ell}} \in {\mathcal{P}}_{\ell}$, there exist 
    ${\ell}'$ with $0 \leq {\ell}' < \ell$ and 
    $p' \in {\mathcal{P}}_{{\ell}'}$, 
    such that $p_1 p_2 \cdots p_{n_{\ell}} \in (p')$. 
  \end{enumerate}
\end{remark}

\begin{proof}[Proof of Theorem $\ref{claim:Bar96_red}$]
It is enough to show $I^{s+1} \subset JI^s$ for some $s \ge 0$. 
\par
For each $\ell = 0, 1, \ldots, r$, 
we set $I_{\ell} = (\mathcal{P}_0 \cup \cdots \cup \mathcal{P}_{\ell})$. 
Then it is enough to prove 
  \begin{equation}
    \label{eq:Bar96_ETS}
    I_{j}^{n_0 n_1 \cdots n_{j}} 
    \subset I_{j-1}^{n_0 n_1 \cdots n_{j-1}} 
            I^{(n_0 n_1 \cdots n_{j -1})
                 (n_{j} - 1)} 
            + J I^{n_0 n_1 \cdots n_{j} - 1}
  \end{equation}
for each $j = 0, 1, \ldots, r$. 

\par
The case of $j = 0$ is clear because $p_0 = g_0 \in J$ 
by the assumption (Ba2). 

\par
Now suppose $j=\ell \geq 1$ and assume that (\ref{eq:Bar96_ETS}) holds 
for every $j \le \ell-1$.  
In order to prove (\ref{eq:Bar96_ETS}) for $j=\ell$, 
it is enough to show that  
for arbitrary $n_0 n_1 \cdots n_{\ell}$ elements 
(to take the same elements is allowed) 
in $\mathcal{P}_0 \cup \cdots \cup \mathcal{P}_{\ell}$, 
the product of these elements is contained in the right hand side 
of (\ref{eq:Bar96_ETS}). 
We divide these elements into  
$n_0 n_1 \cdots n_{\ell - 1}$ sequences of 
$n_{\ell}$ elements, and show that the product of all elements 
in each sequence is contained in 
$I_{\ell - 1} I^{n_{\ell} - 1} + J I^{n_{\ell} - 1}$. 

\par   
In what follows, we discuss about only one sequence. 
If there exists an element of $\mathcal{P}_0 \cup \cdots \cup \mathcal{P}_{\ell-1}$ 
in the sequence, then it is clear that the product 
is contained in $I_{\ell - 1} I^{n_{\ell} -1}$. 
Therefore, we may assume that all elements in the sequence belong to 
${\mathcal{P}}_{\ell}$. 

\par
In the following, we omit the symbol $\ell$ for simplicity. 
Consider the product 
  \begin{displaymath}
    \mu = p_1^{k_1} p_2^{k_2} \cdots p_c^{k_c}, 
    \quad k_1 + k_2 + \cdots + k_c = n, 
    \quad k_i \geq 0 
  \end{displaymath}
and set 
  \begin{displaymath}
    t := t ({\mu}) := \sharp \{ i \; : \; k_i = 1 \}. 
  \end{displaymath}
We prove $\mu \in I_{\ell - 1} I^{n_{\ell} -1}$ 
by descending induction on $t$ ($0 \leq t \leq n$). 

\par
If $t=n$, then $\mu$ is a product of distinct $n$ elements 
in ${\mathcal{P}}_{\ell}$. 
It follows that 
$\mu \in I_{\ell - 1}I^{n_{\ell} - 1}$ by the assumption (Ba3). 

\par
Now we consider the case where $0 \leq t \leq n-1$. 
Then we can assume without loss of generality that 
$k_1 = k_2 = \cdots = k_t =1$ and $k_i \geq 2$ for any $i > t$. 
Notice that $t \le n-2$. 
Let $A'$ be the $(n-1) \times (n-1)$ submatrix of $A$ 
consists of first $n-1$ columns of $A$. 
By assumption, $A'$ is invertible. 
Since $R$ is local, we may assume that it is possible 
to transform the matrix $A$ 
to the matrix $B = (b_{ij})$ having the same size as $A$ 
with $b_{ij} = {\delta}_{ij}$ for $1 \leq i \leq n-1$,  
$1 \leq j \leq n-1$ by elementary row operations. 
Then we put
  \begin{displaymath}
    g_{t+1}' = p_{t+1} + \sum_{j=t+2}^{c} b_{t+1 j} p_j \in J.
  \end{displaymath}
Since $k_{t+1} \geq 2$, we have 
  \begin{displaymath}
    \begin{aligned}
      \mu &= p_1 p_2 \cdots p_t p_{t+1} 
         \bigg(g_{t+1}' - \sum_{j=t+2}^{c} b_{t+1 j} p_j \bigg)^{k_{t+1} - 1}
             p_{t+2}^{k_{t+2}} \cdots p_n^{k_n} \\
          &= g_{t+1}' (\text{an element of $I^{n-1}$}) 
           + p_1 p_2 \cdots p_t p_{t+1} 
         \bigg(- \sum_{j=t+2}^{c} b_{t+1 j} p_j \bigg)^{k_{t+1} - 1}
             p_{t+2}^{k_{t+2}} \cdots p_n^{k_n} \\
          &= (\text{an element of $JI^{n-1}$}) 
           + \sum (\text{an element of $R$}) \cdot
             p_1 p_2 \cdots p_t p_{t+1} \,
             p_{t+2}^{k_{t+2}'} \cdots p_n^{k_n'}, 
    \end{aligned}
  \end{displaymath}
where 
  \begin{displaymath}
    t+1 + k_{t+2}' + \cdots + k_{n}' 
    = t + k_{t+1} +k_{t+2} + \cdots + k_n = n. 
  \end{displaymath}
Then the induction hypothesis implies that the second term in the last 
equation is contained in $I_{\ell - 1} I^{n-1} + J I^{n-1}$. 
This completes the proof. 
\end{proof}

\par
In the next example, the analytic spread of $I$ is known, 
but we can provide a concrete minimal reduction using Theorem \ref{claim:Bar96_red}.  

\begin{exam}
  \label{exam:even_cycle}
Let $r \ge 2$ be an integer. 
Set $I=(x_1 x_2, x_2 x_3, \ldots, x_{2r-1} x_{2r}, x_{2r} x_1)$,  
the edge ideal of the $2r$-cycle $(r \geq 2)$. 
Put  
\begin{displaymath}
    \begin{aligned}
      &{\mathcal{P}}_{\ell} = \{ x_{2 \ell + 1} x_{2 \ell + 2} \}, 
       \quad \ell = 0, 1, \ldots, s-1, \\
      &{\mathcal{P}}_r = \{ x_2 x_3, x_4 x_5, \ldots, 
                            x_{2r- 2} x_{2r- 1}, x_{2r} x_1 \}. 
    \end{aligned}
\end{displaymath}
Then the assumptions of Theorem \ref{claim:Bar96_red} are satisfied 
with $n_{\ell} = 2$ for 
$\ell = 0, 1, \ldots, r-1$ and $n_r = r$. 
Moreover, since all maximal minors of the matrix 
{\small
\[
 A^{(r)} = \left(
\begin{array}{ccccc}
1 &   &       &   & 1 \\
  & 1 &       &   & 1 \\
  &   & \ddots &   & \vdots \\
  &   &       & 1 & 1   
\end{array}
\right)  
\]} 
are unit in $R$, we obtain that  
\[
x_1x_2,\,x_3x_4,\ldots,x_{2r-1}x_{2r},\,x_2x_3+x_{2r}x_1,\,x_4x_5+x_{2r}x_1,\,
\ldots,x_{2r-2}x_{2r-1}+x_{2r}x_{1}
\]
is a reduction of $I$ by Theorem \ref{claim:Bar96_red}. 
\par 
On the other hand, we have $\ell(I)=2r-1$ due to Vasconcelos \cite{Vas}  
because any $2r$-cycle is a bipartite graph.  
In particular, the above reduction is a minimal reduction of $I$. 
\end{exam}

\par 
In the following example, we cannot apply the above theorem, but 
we can find a minimal reduction by a similar argument as in the proof.

\begin{exam}
Let $R=K[x_1,x_2,x_3,x_4,x_5]$ be a polynomial ring over an infinite field $K$, 
and let $a,b,c,d \in K \setminus \{0\}$ be distinct elements with each other. 
Let $I$ be the edge ideal of the complete graph $K_5$, that is, 
$I$ is the ideal generated by the following squarefree monomials of degree $2$:
\[
x_1x_2,\,x_1x_3,\,x_1x_4,\,x_1x_5,\,x_2x_3,\,x_2x_4,\,x_2x_5,\,
x_3x_4,\,x_3x_5,\,x_4x_5.
\]
Set 
\[
\begin{array}{rclcrcl}
\mathcal{P}_0 & = & \{x_1x_2\}, & \qquad &
\mathcal{P}_1 & = & \{x_2x_3,\,x_4x_5 \}, \\
\mathcal{P}_2 & = & \{x_3x_4,\,x_1x_5 \}, & \qquad &
\mathcal{P}_3 & = & \{x_1x_3,\,x_1x_4,\,x_2x_4,\,x_2x_5,\,x_3x_5 \},  
\end{array}
\]
and $I_{\ell} = (\mathcal{P}_0 \cup \cdots \cup \mathcal{P}_{\ell})$ for 
each $\ell = 0,1,2$. 
If we put 
\[
\begin{array}{rcl}
g_0 &=& x_1x_2,\\ 
g_1 &=& x_2x_3+x_4x_5,\\ 
g_2 &=& x_3x_4+x_1x_5,\\ 
g_3 &=& x_1x_3+a x_1x_4+b x_2x_4+ c x_2x_5+d x_3x_5, \\
g_4 &=& x_1x_3+a^2 x_1x_4+b^2 x_2x_4+c^2x_2x_5+d^2 x_3x_5,
\end{array}
\] 
then $J=(g_0,g_1,g_2,g_3,g_4)$ is a (minimal) 
reduction of $I$ by a similar argument as in the proof 
of Theorem \ref{claim:Bar96_red}. 
Indeed, we note that $I_2^3 \subseteq (g_0,g_1,g_2)I^2$.  
\end{exam}

\begin{acknowledgement}
  The first-named author was supported by JST, CREST.
The second-named author was supported by JSPS 20540047. 
The third-named author was supported by JSPS 19340005. 
\end{acknowledgement}

\bibliographystyle{amsplain}

\end{document}